\documentclass[12pt, reqno]{amsart}

\usepackage{amsmath, amssymb, amsthm, amsfonts}

\usepackage[mathscr]{euscript} 
\usepackage{stmaryrd} % St Mary's Road symbols font --- some extra symbols

\usepackage[xcharter]{newtxmath}
\usepackage{graphics, stackrel}
\usepackage{graphicx}

\usepackage{verbatim}
\usepackage{natbib}
\usepackage{enumitem}

\usepackage{caption}
\usepackage{subcaption}

\usepackage{booktabs}

\usepackage{fancyvrb}
\usepackage[dvipsnames,svgnames,table]{xcolor}
\usepackage{mdwlist}

\usepackage[breaklinks=true, citecolor=Brown, colorlinks=true, linkcolor=blue]{hyperref}

\usepackage{enumitem}
\setlist[enumerate]{itemsep=2pt,topsep=3pt}
\setlist[itemize]{itemsep=2pt,topsep=3pt}
\setlist[enumerate,1]{label=(\roman*)}

\usepackage{mathrsfs}  % caligraphic
\usepackage{bbm}
\usepackage{bm}        % bold symbols

\usepackage[left=1.25in, right=1.25in, top=1.0in, bottom=1.15in, includehead, includefoot]{geometry}

\renewcommand{\leq}{\leqslant}
\renewcommand{\geq}{\geqslant}

\providecommand{\inner}[1]{\left\langle{#1}\right\rangle}

\usepackage{enumitem}
\setlist[enumerate]{itemsep=2pt,topsep=3pt}
\setlist[itemize]{itemsep=2pt,topsep=3pt}
\setlist[enumerate,1]{label={\upshape (\roman*)}}

\usepackage[ruled, linesnumbered]{algorithm2e}

\DeclareMathOperator{\fix}{fix}

\DeclareMathOperator*{\argmax}{arg\,max}

\newcommand{\setntn}[2]{ \{ #1 : #2 \} }

\newcommand{\eqdist}{\stackrel {\textrm{ \scriptsize{d} }} {=} }

\newcommand{\1}{\mathbbm 1}

\newcommand*\diff{\mathop{}\!\mathrm{d}}

\newcommand{\bB}{\mathcal B}

\newcommand{\eE}{\mathcal E}
\newcommand{\fF}{\mathscr F}

\newcommand{\pP}{\mathcal P}
\newcommand{\vV}{\mathcal V}

\newcommand{\RR}{\mathbbm R}

\newcommand{\NN}{\mathbbm N}

\newcommand{\TT}{\mathbbm T}

\newcommand{\PP}{\mathbbm P}

\newcommand{\EE}{\mathbbm E}

\newcommand{\Xsf}{\mathsf X}

\newcommand{\Asf}{\mathsf A}
\newcommand{\Gsf}{\mathsf G}

\renewcommand{\phi}{\varphi}
\renewcommand{\epsilon}{\varepsilon}

\theoremstyle{plain}
\newtheorem{theorem}{Theorem}[section]
\newtheorem{corollary}[theorem]{Corollary}
\newtheorem{lemma}[theorem]{Lemma}
\newtheorem{proposition}[theorem]{Proposition}

\theoremstyle{definition}

\newtheorem{example}{Example}[section]

\newtheorem{assumption}{Assumption}[section]

\newcommand{\navy}[1]{\textit{\textcolor{blue}{#1}}}

\begin{document}

\title{Dynamic Programs on Partially Ordered Sets}

\author{John Stachurski}
\address{Research School of Economics, Australian National University}
\email{john.stachurski@anu.edu.au}

\author{Thomas J. Sargent}
\address{Department of Economics, New York University}
\email{ts43@nyu.edu}

\thanks{The authors thank Dimitri Bertsekas, Anmol Bhandari, Shu Hu, Yuchao Li,
    Rami Tabri,
Jingni Yang, and Junnan Zhang for valuable comments and suggestions.}

\begin{abstract}
    We  represent a dynamic program as a family of operators acting on a partially ordered set.  We provide an optimality theory based  on order-theoretic assumptions and show how many applications of dynamic programming fit into this framework.  These range from traditional dynamic programs to those involving nonlinear recursive preferences, desire for robustness, function approximation, Monte Carlo sampling and distributional dynamic programs. We apply our framework to establish new optimality and algorithmic results for specific applications. 
\end{abstract}

\date{\today}

\maketitle

\section{Introduction}\label{s:intro}

Dynamic programs occur in many fields including 
operations research, artificial intelligence, economics, and finance  
\citep{bauerle2011markov, bertsekas2021rollout, kochenderfer2022algorithms}.  They  are used to
price products, control aircraft, sequence DNA, route shipping, recommend
information, and solve frontier research problems.  Within
economics, applications of dynamic programming   range from   monetary and
fiscal policy to asset pricing, unemployment, firm investment, wealth dynamics,
inventory control, commodity pricing, sovereign default, 
natural resource extraction, retirement decisions,
portfolio choice, and dynamic pricing.  

The key idea behind dynamic programming is to reduce an  intertemporal problem with many
stages into a two-period problem by  assigning appropriate
values to future states \citep{bellman1957dynamic}. While optimality theory for
conventional dynamic programs---often called Markov decision processes
(MDPs)---is well understood (see, e.g., \cite{puterman2005markov},
\cite{bauerle2011markov}, \cite{hernandez2012discrete},
\cite{bertsekas2012dynamic}), many recent applications fall outside this
framework.    These include models with nonlinear discounting 
\cite{bauerle2021stochastic}, Epstein--Zin preferences
\citep{epstein1989risk}, risk-sensitive preferences \citep{hansen1995discounted,
bauerle2018stochastic}, adversarial agents \citep{cagetti2002robustness,
hansen2011robustness}, and ambiguity \citep{iyengar2005robust, yu2024fast}.   

Mathematicians have constructed  frameworks that  include both
standard MDPs and the growing list of nonstandard applications discussed in the
last paragraph. \cite{bertsekas2022abstract} contributed important work in this direction.  His ``abstract dynamic programming'' framework
extends earlier ideas dating back to \cite{denardo1967contraction} and
\cite{bertsekas1977monotone} and
generalizes the traditional Bellman equation in ways  that can represent  many
different models. Recent applications of the framework include
\cite{ren2021dynamic}, \cite{bloise2024not} and \cite{toda2023unbounded}.

Nevertheless, many  classes of dynamic programs  lie
outside the framework of \cite{bertsekas2022abstract}.  One such class is 
dynamic programs that reverse the order of expectation and maximization in the
Bellman equation, with well known examples in optimal stopping
\citep{jovanovic1982selection, hubmer2020sources},
Q-learning (see, e.g., \cite{watkins1989learning, kochenderfer2022algorithms}), and
structural estimation (\cite{rust1987optimal}, \cite{rust1994structural}, and
\cite{mogensen2018solving}, etc.). Another such class involves dynamic programs with
value functions that are not real-valued. Examples can be
found in distributional dynamic programming \citep{bellemare2017distributional}, where states are mapped to distributions, and in
empirical dynamic programming, where states are mapped to 
random elements taking values in a function space
\citep{munos2008finite, haskell2016empirical, bertsekas2021rollout,
kalathil2021empirical, rust2002there, sidford2023variance}. A related situation
occurs with  dynamic programs cast  in $L_p$ spaces, where value ``functions''
are actually  equivalence classes of functions.

Another motivation for extending the framework of \cite{bertsekas2022abstract}
is to study approximate dynamic programming, which replaces exact
value and policy functions with parametric or nonparametric approximations
\citep{powell2016perspectives, bertsekas2021rollout, bertsekas2022abstract}.
The framework developed here facilitates viewing at least some of these approximate dynamic
programs as dynamic programs in their own right.  One advantage  is acquiring the ability to analyze whether an  approximate dynamic program under study 
 is a well-defined Bellman-type optimization problem, in the sense that a single
policy function is optimal starting from all possible initial states (see, e.g.,
\cite{naik2019discounted}).  

In order to capture the broad class of applications described above,
 the framework developed in this paper replaces 
traditional value functions  with elements of an abstract
partially ordered set $V$. Policies that are traditionally understood as
mappings from states to actions now  become abstract indices over a family $\setntn{T_\sigma}{\sigma \in \Sigma}$ of
``policy operators.''  Each
policy operator $T_\sigma$ is an order preserving self-map on $V$.   Whenever it exists, the lifetime
value $v_\sigma$ of a policy $\sigma \in \Sigma$ is identified with the unique
fixed point of $T_\sigma$. A policy $\sigma^*$ is defined to
be optimal when $v_{\sigma^*}$ is a greatest element of
$\setntn{v_\sigma}{\sigma \in \Sigma}$. 
In this setting we provide an order-theoretic treatment of dynamic programming
and describe  conditions under which some classical dynamic programming optimality
results hold, e.g., the value function is a unique fixed point of the Bellman
equation, Bellman's principle of optimality is valid.  We also provide  conditions
under which standard  dynamic programming algorithms converge.  

Our framework brings three significant benefits.  First,  policy operators are  general enough 
to represent both standard and many nonstandard dynamic programs.  Second, because we work in an abstract partially ordered space, it is
possible to handle Bellman equations defined not only over spaces of real-valued
functions, but also over spaces of distributions and spaces of random functions. Third, its high level
of abstraction simplifies analysis and  isolates  roles of various assumptions.
We illustrate our  framework by applying it to  establish new results in the context of several applications.
.

Section~\ref{s:prelim} introduces terminology and essential concepts.
Section~\ref{s:adps} introduces our abstract dynamic programming framework  and provides  examples that help illustrate the definitions.  Section~\ref{s:mo} defines
optimality and Section~\ref{s:opres} provides fundamental optimality results that  are then proved in Section~\ref{s:ir}.  
Section~\ref{s:apps} presents applications  and
Section~\ref{s:con} concludes.

\section{Preliminaries}\label{s:prelim}

Let $V = (V, \preceq)$ be a partially ordered set.  We use the symbol $\bigvee$
to represent suprema; for example, if $\{v_\alpha\}_{\alpha \in \Lambda}$ is a
subset of $V$, then $\bigvee_\alpha v_\alpha$ is the least element of the set of
upper bounds of $\{v_\alpha\}_{\alpha \in \Lambda}$.  $V$ is called \navy{bounded}
when $V$ has a least and greatest element.  A subset $C$ of $V$ is called
a \navy{chain} if it is totally ordered by $\preceq$. A sequence $(v_n)_{n \in
\NN}$ in $V$ is called \navy{increasing} if $v_n \preceq v_{n+1}$ for all $n \in
  \NN$. If $(v_n)$ is increasing and $\bigvee_n v_n = v$ for some $v \in V$,
  then we write $v_n \uparrow v$.
A partially ordered set $V$ is called \navy{chain complete}
(resp., \navy{countably chain complete}) if $V$ is bounded
and every chain (resp., every increasing sequence) in $V$ has a supremum.  $V$
is called \navy{countably Dedekind complete} if every 
countable subset of $V$ that is bounded above (i.e., the set of upper bounds is
nonempty) has a supremum in $V$.

A self-map $S$ on $V$ is called
\navy{order preserving} on $V$ when $v, w \in V$ and $v \preceq w$
        imply $Sv \preceq Sw$, and
\navy{order continuous} on $V$ if, for any $(v_n)$ in $V$ with $v_n
        \uparrow v$, we have $S v_n \uparrow S v$.\footnote{The definition of
        order continuity varies across subfields of mathematics but the one just
        given suffices for our purposes.} 
Simple arguments show that if $S$ is order continuous on $V$ then $S$ is order preserving.

In the theorem below, $S$ is a self-map on partially ordered set $V$ and
$\fix(S)$ is the set of all fixed points of $S$ in $V$. While the result is
well-known, our version is slightly nonstandard so we include a partial proof.

\begin{theorem}\label{t:tk}
    The set $\fix(S)$ is nonempty if either
    \begin{enumerate}
        \item $V$ is chain complete and $S$ is order preserving, or
        \item $V$ is countably chain complete and $S$ is order continuous.
    \end{enumerate}
     Moreover, in the second case, 
        $v \in V$ and $v \preceq Sv$ implies
        $\bigvee_n S^n v \in \fix(S)$.
\end{theorem}

\begin{proof}
    For a proof of case (i), see, for example, Theorems~8.11 and 8.22 of
    \cite{davey2002introduction}.  As for (ii), let $S, V$ be as
    stated in (ii) and fix $v \in V$ with $v \preceq Sv$.   $S$ is order
    continuous and hence order preserving, so $(v_n) \coloneq (S^n v)$ is
    increasing. As $V$ is countably chain complete,
    the suprema $\bigvee_{n \geq 1} v_n$ and $\bigvee_{n \geq 1} S v_n$ exist
    in $V$.  If $\bar v \coloneq \bigvee_n v_n$, then
        $S \bar v 
        = S \bigvee_{n \geq 1} v_n
        = \bigvee_{n \geq 1} S v_n
        = \bigvee_{n \geq 2} v_n
        = \bar v$, where the second equality is by order continuity.
    Hence $\bar v \in \fix(S)$. Finally, countable chain completeness implies
    that $V$ has a least element $\bot$.  We then have $\bot \preceq S \bot$, so
    $\fix(S)$ is nonempty.
\end{proof}

We call a self-map $S$ on $V$ \navy{upward stable} if $S$ has a unique fixed
point $\bar v$ in $V$ and $v \preceq S v$ implies $v \preceq \bar v$,
\navy{downward stable} if $S$ has a unique fixed point $\bar v$ in $V$ and $S  v
\preceq v$ implies $\bar v \preceq v$,  and \navy{order stable} if $S$ is both
upward and downward stable. Order stability is a weak and purely order-theoretic
version of asymptotic stability.

\begin{example}\label{eg:pospace}
    Let $V$ be  a metric space endowed with a
    closed partial order $\preceq$, so that $u_n \preceq v_n$ for all $n$
    implies $\lim_n u_n \preceq \lim_n v_n$ whenever the limits exist,
    and let $S \colon V \to V$ be order preserving and
    globally stable under the metric on $V$, so that $S$ has a unique fixed point $\bar v$ in
    $V$ and $\lim_n S^n v = \bar v$ for all $v \in V$.  If $v \preceq Sv$, then 
    $v \preceq S^n v$ for all $n$.  Taking the limit and using the fact that
    $\preceq$ is closed yields $v \preceq \bar v$.  Hence $S$ is upward stable.
    A similar argument shows that $S$ is downward stable.
\end{example}

\begin{lemma}\label{l:ocius}
    Let $S$ be a self-map on $V$ with at most one fixed point in $V$. If either
    \begin{enumerate}
        \item $V$ is chain complete and $S$ is order preserving, or
        \item $V$ is countably chain complete and $S$ is order continuous, 
    \end{enumerate} 
    then $S$ is order stable on $V$.
\end{lemma}

\begin{proof}
    First suppose that $S$ is order preserving and $V$ is chain complete, with least element $\bot$ and
    greatest element $\top$. By Theorem~\ref{t:tk}, $S$ has a fixed point $\bar v
    \in V$.  By assumption, $\bar v$ is the only fixed point of $S$ in $V$.  Now
    fix $v \in V$ with $v \preceq Sv$. Since $I \coloneq [v, \top]$ is itself
    chain complete, and since $S$ maps $I$ to itself, $S$ has a fixed point in
    $I$.  Hence $\bar v \in I$, which yields $v \preceq \bar v$.  This proves
    upward stability.  The proof of downward stability is similar.

    Now suppose that $V$ is countably chain complete and $S$ is order
    continuous. Let $\bar v$ be the unique fixed point of $S$ in $V$ (existence
    of which follows from Theorem~\ref{t:tk}). Pick any $v \in V$
    with $v \preceq Sv$.  Since $I \coloneq [v, \top]$ is itself countably
    chain complete, and since $S$ is an order continuous map from $I$ to itself,
    Theorem~\ref{t:tk} implies that $S$ has a fixed point in $I$.
    Hence $\bar v \in I$, which yields $v \preceq \bar v$.  This proves upward
    stability.  The proof of downward stability is similar.
\end{proof}

\section{Abstract Dynamic Programs}\label{s:adps}

We define an \navy{abstract dynamic program} (ADP) to be a pair $(V, \TT)$, where 
$V = (V, \preceq)$ is a partially ordered set
and $\TT = \setntn{T_\sigma}{\sigma \in \Sigma}$ is a family of
        order preserving self-maps on $V$. 
The set $V$ is called the \navy{value space}.
The operators in $\TT$ are called \navy{policy operators}.
$\Sigma$ is an arbitrary index set and elements of $\Sigma$ 
will be referred to as \navy{policies}. 
In applications we impose conditions under which each $T_\sigma$ has a unique
fixed point.  In these settings, the significance of $T_\sigma$ is that its
fixed point, denoted below by $v_\sigma$, represents the lifetime value (or
cost) of following policy $\sigma$. 

\begin{example}
    In some settings, $V$ is a set of functions and $\preceq$ is
    the pointwise partial order $\leq$ (i.e., $v \leq w$ if $v(x) \leq w(x)$ for all
    $x$ in $\Xsf$).  The value $v_\sigma(x)$ represents the lifetime
    value of following policy $\sigma$ when the initial state is $\Xsf$.
\end{example}

Let $(V, \TT)$ be an ADP with policy set $\Sigma$. Given $v \in V$, a policy
$\sigma$ in $\Sigma$ is called \navy{$v$-greedy} if $T_{\sigma} \, v \succeq
T_\tau \, v$ for all $\tau \in \Sigma$.  We call $(V, \TT)$ \navy{regular} when
each $v \in V$ has at least one $v$-greedy policy.  Given $v \in V$, we set
\begin{equation}\label{eq:adp_bellop}
    Tv = \bigvee_{\sigma \in \Sigma} T_\sigma \, v
\end{equation}
whenever the supremum exists. We call $T$ the \navy{Bellman operator} generated
by $(V, \TT)$.   We say that $v \in V$ \navy{satisfies the Bellman equation} if $Tv = v$.

For a given ADP $(V, \TT)$, we define three sets:
\begin{itemize}
    \item $V_G \coloneq \setntn{v \in V}{\text{ at least one $v$-greedy policy exists}}$,
    \item $V_\Sigma \coloneq \setntn{v \in V}{\text{ $v$ is a fixed point of
            $T_\sigma$ for some $\sigma \in \Sigma$}}$, and
    \item $V_U \coloneq \setntn{v \in V}{ \, v \preceq Tv}$.
\end{itemize}
The next lemma shows that $T$ has attractive properties on $V_G$.
  
\begin{lemma}\label{l:torper}
    The Bellman operator $T$ has the following properties:
    \vspace{0.4em}
    \begin{enumerate}
        \item $T$ is well-defined and order preserving on $V_G$.    
        \item For $v \in V_G$ we have $T_\sigma \, v = T v$ if and only if $\sigma \in \Sigma$ is $v$-greedy.
    \end{enumerate}
\end{lemma}

\begin{proof}
    We begin with part (ii).  Fix $v \in V_G$ and let $\sigma$ be $v$-greedy.
    Then, by definition, $T_\sigma \, v$ is the greatest element of $\{T_\tau \,
    v\}_{\tau \in \Sigma}$. A greatest element is also a supremum, so $T v  =
    T_\sigma \, v$.  Conversely, if
    $Tv = T_\sigma \, v$, then $T_\tau \, v \preceq T_\sigma \, v$ for all $\tau
    \in \Sigma$.  Hence (ii) holds.  As for (i),
    fixing $v \in V_G$, a $v$-greedy policy exists, so $Tv$ is
    well-defined by part (ii). As for the order preserving claim, fix
    $v, w \in V_G$ with $v \preceq w$. Let $\sigma \in \Sigma$ be $v$-greedy.
    Since $T_\sigma$ is order preserving, we have $T v = T_\sigma \, v \preceq
    T_\sigma \, w \preceq T w$.  
\end{proof}

\begin{lemma}\label{l:vsigvu}
    If $(V, \TT)$ is regular, then $V_\Sigma \subset V_U$.
\end{lemma}

\begin{proof}
    Fix $v_\sigma \in V_\Sigma$ and let $T_\sigma$ be such that $v_\sigma$ is a
    fixed point.  Since $T$ is well-defined on all of $V$, we have
    $v_\sigma = T_\sigma \, v_\sigma \preceq T v_\sigma$.  In particular, $v_\sigma \in V_U$.
\end{proof}

\subsection{Example: MDPs}\label{ss:mdps}

We begin with a simple and familiar example involving 
Markov decision processes (MDPs,
see, e.g., \cite{bauerle2011markov}).  The objective is to maximize 
$\EE \sum_{t \geq 0} \beta^t r(X_t, A_t)$
where $X_t$ takes values in finite set $\Xsf$ (the state space),
$A_t$ takes values in finite set $\Asf$ (the action space),
$\Gamma$ is a nonempty correspondence from $\Xsf$ to $\Asf$
(the feasible correspondence),
$\Gsf \coloneq \setntn{(x, a) \in \Xsf \times \Asf}{a \in \Gamma(x)}$
        denotes the feasible state-action pairs,
    $r$ is a reward function defined on $\Gsf$, 
    $\beta \in (0,1)$ is a discount factor, and 
    $P \colon \Gsf \times \Xsf \to [0,1]$ provides transition
        probabilities. The Bellman equation  is 
\begin{equation}\label{eq:mdp_bell}
    v(x) = \max_{a \in \Gamma(x)} 
    \left\{
        r(x, a) + \beta \sum_{x'} v(x') P(x, a, x')
    \right\}
    \qquad\qquad (x \in \Xsf).
\end{equation}
The set of feasible policies is the finite set $\Sigma \coloneq \setntn{\sigma
\in \Asf^\Xsf} {\sigma(x) \in \Gamma(x) \text{ for all } x \in \Xsf}$. We
combine $\RR^\Xsf$ (the set of all real-valued functions on $\Xsf$) with the
pointwise partial order $\leq$ and, for $\sigma \in \Sigma$ and $v \in
\RR^\Xsf$, define the MDP policy operator
\begin{equation}\label{eq:tsig_mdp}
    (T_\sigma \, v)(x) 
        = r(x, \sigma(x)) + \beta \sum_{x'} v(x') P(x, \sigma(x), x')
        \qquad\qquad (x \in \Xsf).
\end{equation}
Let $V \coloneq  \RR^\Xsf$ and $\TT = \setntn{T_\sigma}{\sigma \in \Sigma}$. Since each $T_\sigma$ is an
order preserving self-map on $V$, the pair $(V, \TT)$ is an ADP. 
Given $v \in V$, let $\sigma \in \Sigma$ be such that
\begin{equation}\label{eq:mdpmvg}
    \sigma(x) \in \argmax_{a \in \Gamma(x)}
    \left\{
        r(x, a) + \beta \sum_{x'} v(x') P(x, a, x')
    \right\}
    \quad \text{for all $x \in \Xsf$}.  
\end{equation}
Such a $\sigma$ satisfies $T_\sigma \, v \geq T_\tau \, v$ for all $\tau \in
\Sigma$, implying $\sigma$ is $v$-greedy. Moreover, since $\Gamma$ is nonempty,
at least one policy obeying \eqref{eq:mdpmvg} exists.  This proves that
$(V, \TT)$ is regular.

We stated above that, in applications, the fixed point of each
$T_\sigma \in \TT$ is the lifetime value of policy
$\sigma$.  To see the idea in the MDP setting, fix $\sigma \in \Sigma$ and let 
$r_\sigma$ and $P_\sigma$ be defined by 
$P_\sigma(x, x') \coloneq P(x, \sigma(x), x')$ and 
    $r_\sigma(x) \coloneq r(x, \sigma(x))$.
In the present setting, the lifetime value of $\sigma$ given $X_0 = x$ is
understood to be $v_\sigma(x) = \EE \sum_{t \geq 0} \beta^t r_\sigma(X_t)$, where $(X_t)_{t \geq 0}$ is a Markov chain generated by
$P_\sigma$ with initial condition $X_0 = x \in \Xsf$. Pointwise on $\Xsf$, we
can express $v_\sigma$ as $v_\sigma = \sum_{t \geq 0} (\beta P_\sigma)^t r_\sigma
            = (I-\beta P_\sigma)^{-1} r_\sigma$
(see, e.g., \cite{puterman2005markov}, Theorem~6.1.1). This implies that $v_\sigma$
is the unique solution to the equation $v = r_\sigma + \beta P_\sigma \, v$.
From the definition of $T_\sigma$ in \eqref{eq:tsig_mdp}, this is
equivalent to the statement that $v_\sigma$ is the unique fixed point of
$T_\sigma$.

For the ADP $(V, \TT)$, the ADP Bellman equation
\eqref{eq:adp_bellop} reduces to the MDP Bellman equation \eqref{eq:mdp_bell}.
This connection is important because it allows to use optimality properties of $(V, \TT)$  to study optimality properties of the MDP.  To see that it holds observe that,
since $V$ is endowed with the pointwise partial order,
for given $v \in V$ and $x \in \Xsf$,
the ADP Bellman
operator \eqref{eq:adp_bellop} reduces to 
\begin{equation*}
    (T \, v)(x)
    = \sup_{\sigma \in \Sigma} (T_\sigma \, v)(x)
    = \sup_{\sigma \in \Sigma} 
    \left\{
        r(x, \sigma(x)) + \beta \sum_{x'} v(x') P(x, \sigma(x), x')
    \right\}.
\end{equation*}
By the definition of $\Sigma$, we can also write this as
\begin{equation}\label{eq:mdpt}
    (T \, v)(x) = \max_{a \in \Gamma(x)} 
    \left\{
        r(x, a) + \beta \sum_{x'} v(x') P(x, a, x')
    \right\}.
\end{equation}
Evidently $v$ satisfies the
ADP Bellman equation $T v = v$ if and only if the traditional MDP Bellman
equation \eqref{eq:mdp_bell} holds.

\subsection{Example: Risk-Sensitive Q-learning}\label{ss:rsqfac}

Some dynamic programs reverse the order of maximization and mathematical expectation.
One example is Q-factor risk-sensitive decision processes (see, e.g.,
\cite{fei2021exponential}), where the Bellman equation takes the form 
\begin{equation}\label{eq:rsqbell}
    f(x, a) = 
        r(x, a) + \frac{\beta}{\theta}
        \ln 
        \left\{ 
                \sum_{x'} \exp \left[ \theta \max_{a' \in \Gamma(x')} f(x', a')
            \right] P(x, a, x')
        \right\}
\end{equation}
for $(x,a) \in \Gsf$ and nonzero $\theta \in \RR$.
(We take $\Xsf$, $\Asf$, $\Gamma$, $\Sigma$ and $\Gsf$ as in our discussion of
MDPs in Section~\ref{ss:mdps}.) Given $\sigma \in \Sigma$, the corresponding policy operator is
\begin{equation}\label{eq:rspbell}
    (T_\sigma f)(x, a) = 
        r(x, a) + \frac{\beta}{\theta}
        \ln \left[ 
                \sum_{x'} \exp \left[ \theta f(x', \sigma(x'))
            \right] P(x, a, x')
        \right]
\end{equation}
where $f \in \RR^\Gsf \coloneq$ all real-valued functions on $\Gsf$. With $\TT
\coloneq \setntn{T_\sigma}{\sigma \in \Sigma}$ and $\Gsf$ endowed with the
pointwise order, the pair $(\RR^\Gsf, \TT)$ is an ADP.  
The ADP Bellman operator is $T f \coloneq \bigvee_\sigma T_\sigma \, f$ (by
\eqref{eq:adp_bellop}) and the ADP Bellman equation is $Tf = f$.
By replicating arguments in Section~\ref{ss:mdps},
one can show that $f \in \Gsf$ solves $Tf=f$ if and only if it
solves \eqref{eq:rsqbell}.  This allows us to study optimality properties
of the original model (as characterized by \eqref{eq:rsqbell}) through the
ADP $(\RR^\Gsf, \TT)$.

\subsection{Distributional Dynamic Programming}\label{ss:ddp}

Distributional dynamic programming focuses on an  entire distribution of  lifetime returns, not just  its expected value
\citep{bellemare2017distributional}. This 
falls outside frameworks such as \cite{puterman2005markov} or
\cite{bertsekas2022abstract} because elements of the value space $\pP^\Xsf$ are
not real-valued functions.  In this section we show how  distributional dynamic
programming  can be represented in the setting of ADPs.

In what follows, $\pP$ is the set of all probability distributions on
$\RR$ and $\pP^\Xsf$ is the set of all functions from $\Xsf$ into $\pP$. A
typical element is written as $\eta(x, \diff g)$, indicating that $\eta(x,
\cdot)$ is a distribution on $\RR$ for each $x \in \Xsf$.
For $\eta, \hat \eta \in \pP^\Xsf$, we write $\eta \preceq \hat \eta$ when
$\eta(x)$ lies below $\hat \eta(x)$ in the sense of stochastic dominance for
every $x \in \Xsf$.  Thus,
\begin{equation*}
    \eta \preceq \hat \eta
    \quad \iff \quad
    \int h(g) \eta(x, \diff g)
    \leq \int h(g) \hat \eta(x, \diff g)
    \quad \forall \, h \in ib\RR, \; x \in \Xsf,
\end{equation*}
where $ib\RR$ is the set of increasing bounded measurable functions from $\RR$
to itself.

Maintaining the MDP setting in Section~\ref{ss:mdps} but switching to a
distributional perspective, the distributional policy operator, written here as
$D_\sigma$, maps $\eta \in \pP^\Xsf$ to $D_\sigma \eta$, where $(D_\sigma
\eta)(x)$ is the distribution of the random variable $G' \coloneq r_\sigma(x) + \beta
G_{X'}$ when $G_{X'}$ is sampled by first drawing the next period state $X'$ from
$P_\sigma(x, \cdot)$ and then drawing $G$ from $\eta(X', \cdot)$.  So  the
expectation of $h \in ib\RR$ under the distribution $(D_\sigma \eta)(x)$ can be
expressed as
\begin{equation*}
    \EE h(G')
    \coloneq \inner{h, (D_\sigma \eta)(x)}
    \coloneq
    \sum_{x'} \int h(r_\sigma(x) + \beta g) \eta(x', \diff g) P_\sigma(x, x').
\end{equation*}
If we now take $\eta \preceq \hat \eta$ and $h \in ib\RR$, we get
    $\int h(r_\sigma(x) + \beta g) \eta(x', \diff g) 
    \leq \int h(r_\sigma(x) + \beta g) \hat \eta(x', \diff g) $
for any $x'$ and hence $\inner{h, (D_\sigma \eta)(x)} \leq \inner{h, (D_\sigma
\hat \eta)(x)}$.  Since this holds for any $x$ we have $D_\sigma \, \eta \preceq
D_\sigma \, \hat \eta$, so $D_\sigma$ is order preserving.  In particular,
$(\pP^\Xsf, \{D_\sigma\}_{\sigma \in \Sigma})$ is an ADP.

\subsection{Empirical Dynamic Programming}\label{ss:edp}

Monte Carlo estimates are sometimes used to approximate mathematical expectations in dynamic programs with 
 large state spaces. For example, in
\cite{haskell2016empirical}, the MDP Bellman operator \eqref{eq:mdpt} is
replaced by 
\begin{equation*}
    (\hat T \, v)(x) = \max_{a \in \Gamma(x)} 
    \left\{
        r(x, a) + \beta \frac{1}{n} \sum_{i=1}^n v(F(x, a, \xi_i))
    \right\},
\end{equation*}
where $(\xi_i)_{i=1}^n$ is a collection of random variables on probability space
$(\Omega, \fF, \PP)$ and each $F(x, a, \xi_i)$ has distribution $P(x, a,
\cdot)$. The corresponding policy operators $\hat \TT \coloneq \setntn{\hat
T_\sigma}{\sigma \in \Sigma}$ are %
\begin{equation}\label{eq:tsige}
    (\hat T_\sigma \, v)(x) 
    = r(x, \sigma(x)) + \beta \frac{1}{n} \sum_{i=1}^n v(F(x, \sigma(x), \xi_i))
    .
\end{equation}
Following \cite{haskell2016empirical}, we take $\vV$ to be the set of random
elements defined on probability space $(\Omega, \fF, \PP)$ and taking values in
the function space $\RR^\Xsf$.  To make the dependence on $\omega \in \Omega$ explicit we write a
realization as $v(\omega, \cdot)$, so that $x \mapsto v(\omega, x)$ is a
function in $\RR^\Xsf$ assigning values to states.  The policy operator
\eqref{eq:tsige} then becomes
\begin{equation}\label{eq:tsige2}
    (\hat T_\sigma \, v)(\omega, x) 
    = r(x, \sigma(x)) + \beta \frac{1}{n} \sum_{i=1}^n v(\omega, F(x, \sigma(x), \xi_i(\omega)))
    .
\end{equation}
A partial order can be introduced on $\vV$ by writing $v \leq w$ when $v(\omega,
x) \leq w(\omega, x)$ for all $\omega \in \Omega$ and $x \in \Xsf$.
It is clear that $\hat T_\sigma \, v \leq \hat T_\sigma \, w$ whenever $v \leq
w$, so $(\vV, \hat \TT)$ is an ADP.

\subsection{Example: Approximate Dynamic Programming}\label{ss:approx}

In practice, solution methods for a vast range of dynamic programs involve some
form of function approximation to simplify update steps and
generate representations of value and policy functions (see, e.g.,
\cite{powell2016perspectives, bertsekas2021rollout, bertsekas2022abstract}). For
example, the MDP policy operator $T_\sigma$ from \eqref{eq:tsig_mdp} might be
replaced by $A \circ T_\sigma$, where $A$ implements an approximation
architecture such as kernel averaging or projection onto a  space of basis functions.

When function approximation is added to policy and Bellman operators, properties needed to pose a 
well-defined dynamic program may not be satisfied.  For example, function approximations may transmute a well-behaved MDP
 into a dynamic program for which  no optimal stationary policy
exists \citep{naik2019discounted}.  
Our %ADP 
framework % developed in the present paper 
facilitating  addressing 
these issues.  For example, if the approximation operator $A$ is order
preserving, then setting $V = \RR^\Xsf$ and $\TT_A = \setntn{A \circ
T_\sigma}{\sigma \in \Sigma}$ yields an ADP $(V, \TT_A)$. Below we provide
results under which ADPs such as $(V, \TT_A)$ have well-defined optimal
policies.

\subsection{Example: Structural Estimation}\label{ss:struct}

\cite{rust1987optimal} and many subsequent authors
study discrete choice problems with Bellman equations
of the form 
\begin{equation}\label{eq:stbe}
    g(x, a) =
    \sum_{x'} \int
        \left\{ 
            \max_{a' \in \Asf}
            \left[
                r(x', a', e') + \beta g(x', a')
            \right]
        \right\} \nu(\diff e') P(x, a, x').
\end{equation}
Here $(x,a) \in \Gsf \coloneq \Xsf \times \Asf$ where $\Asf$ and $\Xsf$ are
the action and state spaces respectively.   The set $\Asf$ is finite (hence
discrete choice) and we take $\Xsf$ to be finite for simplicity (although other
settings can also be handled). We assume that $r$ is a bounded reward function,
$\beta \in (0,1)$ and $P(x, a, \cdot)$ is a probability distribution
(probability mass function) on $\Xsf$ for each $(x,a) \in \Gsf$. The component
$e'$ in the reward function takes values in some measurable space and is {\sc
iid} with distribution $\nu$.

The function $g$ can be  interpreted as an ``expected post-action value function.''
Advantages of working with this version of the Bellman equation are
discussed in \cite{rust1994structural}, \cite{kristensen2021solving} and other
sources.  Because the max operator is inside the expectation, frameworks such as
\cite{puterman2005markov} and \cite{bertsekas2022abstract} do not directly
apply. Nevertheless, we can set this problem up as an ADP by taking $\Sigma$
to be the set of maps from $\Xsf$ to $\Asf$ and, for each $\sigma \in \Sigma$,
introducing the policy operator
\begin{equation}\label{eq:ruts}
    (T_\sigma \, g)(x, a) =
    \sum_{x'}
        \int
        \left\{ 
                r(x', \sigma(x'), e') + \beta g(x', \sigma(x'))
        \right\} \nu(\diff e') P(x, a, x').
\end{equation}
Clearly $T_\sigma$ maps $\RR^\Gsf$ into itself and is order preserving on
$\RR^\Gsf$ under the pointwise partial order.  Hence, with $\TT
\coloneq \setntn{T_\sigma}{\sigma \in \Sigma}$, the pair  $(\RR^\Gsf, \TT)$ is an ADP.
Moreover, if we take $M \in \NN$
such that $|r| \leq M$ and set $W$ to all $g \in \RR^\Gsf$ with $|g| \leq
M/(1-\beta)$, then straightforward calculations show that $T_\sigma$ maps $W$ to
itself.  Hence $(W, \TT)$ is also an ADP.

\section{Properties of ADPs}\label{s:mo}

In this section we define optimality for ADPs.  We also categorize ADPs with
the aim of determining properties that lead to strong optimality results.

\subsection{Basic Properties}\label{ss:props}

Let $(V, \TT)$ be an ADP with policy set $\Sigma$. When it exists, we denote the
unique fixed point of $T_\sigma$ in $V$ by $v_\sigma$. We call $(V, \TT)$
\navy{finite} if $\TT$ is finite,
\navy{well-posed} if each $T_\sigma \in \TT$ has a unique fixed point
        $v_\sigma$ in $V$,
\navy{order stable} if each $T_\sigma \in \TT$ is order stable on $V$, 
\navy{order continuous} if each $T_\sigma \in \TT$ is order continuous
        on $V$, and \navy{bounded above} if there exists a $u \in V$ with
$T_\sigma \, u \preceq u$ for all  $T_\sigma \in \TT$.

In our applications, the lifetime value of a policy $\sigma$
coincides with the fixed point $v_\sigma$ of its policy operator $T_\sigma$.
Well-posedness is a minimal  condition because without it we cannot be sure that
policies have well-defined lifetime values. Maximizing lifetime values (or,
equivalently, minimizing lifetime costs) is the objective of dynamic programming.

\begin{example}
    The distributional ADP $(\pP^\Xsf, \{D_\sigma\}_{\sigma \in \Sigma})$ described in Section~\ref{ss:ddp} 
    is well-posed, since, given boundedness of the reward function $r$ and $\beta
    \in (0,1)$, each policy operator $D_\sigma$ has a unique fixed point in
    $\pP^\Xsf$.  This existence and uniqueness result follows from Theorem~1 of
    \cite{gerstenberg2023solutionsdistributionalbellmanequation}.
\end{example}

\begin{example}\label{eg:mdpos}
    The ADP $(V, \TT)$ generated by the MDP model in
    Section~\ref{ss:mdps} is finite, well-posed, order stable, order
    continuous, and bounded above. Finiteness holds because $\Xsf$ and $\Asf$ are finite.  $(V, \TT)$ is well-posed
    because, under the stated assumptions, each $T_\sigma \, v = r_\sigma + \beta P_\sigma \, v$ has a unique fixed point in $V$ given by $v_\sigma
    \coloneq (I-\beta P_\sigma)^{-1} r_\sigma$.  Order stability follows from 
    Lemma~\ref{l:ocius}.  $(V, \TT)$ is order continuous
    because $v_n \uparrow v$ is equivalent to 
    $v_n \to v$ pointwise when $V = \RR^\Xsf$.  Hence
    $v_n \uparrow v$ implies $T_\sigma \, v_n \uparrow T_\sigma \, v$.
    Finally, $(V, \TT)$ is bounded above because, for any $\sigma \in \Sigma$,
    \begin{equation*}
        u  = \frac{\max r }{1 - \beta}
        \quad \; \implies \; \quad
        T_\sigma \, u 
            = r_\sigma + \beta P_\sigma \, u
            \leq \max r + \beta u = u.
    \end{equation*}
\end{example}

\begin{example}\label{eg:structos}
    The ADP $(W, \TT)$ generated by the dynamic structural model in
    Section~\ref{ss:struct} is finite, well-posed, order stable, and order
    continuous. The proof is almost identical to that given in
    Example~\ref{eg:mdpos}.
\end{example}

  % Order continuity holds because $\Xsf$
  %   is finite, so $v_n \uparrow v$ if and only if $(v_n)$ increases to $v$
  %   pointwise (\cite{aliprantis1999border}, Lemma~7.16).

Below we use order stability as a condition for optimality. The next lemma shows
that order continuity passes from the policy operators to the Bellman operator.

\begin{lemma}\label{l:ostst}
    If $(V, \TT)$ is order continuous and $V$ is countably chain complete, then
    $T$ is order continuous on $V$.
\end{lemma}

\begin{proof}
    Fix $(v_n) \subset V$ with $v_n \uparrow \bar v \in V$. Since, $T$ is
    order preserving, $(T v_n)_{n \geq 1}$ is also increasing. Hence $\bigvee_n T
    v_n$ exists in $V$.  We claim that $\bigvee_n T v_n = T \bar v$.  
    On one hand, $T \bar v$ is an upper bound for $(T v_n)$. On the other hand, if $w
    \in V$ is such that $T v_n \preceq w$ for all $n$, then $T_\sigma \, v_n
    \preceq w$ for all $n$ and $\sigma$.  Fixing $\sigma \in \Sigma$, taking the
    supremum over $n$ and using order-continuity of $T_\sigma$ gives $T_\sigma
    \, \bar v \preceq w$. Hence $T \bar v \preceq w$, which means that $T \bar
    v$ is a least upper bound of $(T v_n)$.  This confirms that $\bigvee_n T v_n
    = T \bar v$. Hence $T$ is order continuous. 
\end{proof}

\subsection{Defining Optimality}\label{ss:defop}

Next we define optimality for ADPs  using concepts that are direct
generalizations of dynamic program optimality from existing frameworks.
To begin, we recall that, for a well-posed ADP
$(V, \TT)$ with policy set $\Sigma$, the symbol 
$V_\Sigma$ represents the set of lifetime values $\{v_\sigma\}_{\sigma \in
\Sigma}$ generated by $(V, \TT)$. A policy $\sigma \in \Sigma$ is called
\navy{optimal} for $(V, \TT)$ if $v_\sigma$ is a greatest element of $V_\Sigma$.

\begin{example}
    The ADP $(V, \TT)$ generated by the MDP model in
    Section~\ref{ss:mdps} uses the pointwise partial order on $V$,
    so $v_\sigma$ is optimal if and only if $v_\sigma(x) = \max_{\tau \in
    \Sigma} v_\tau(x)$ for all $x \in \Xsf$.  This is the standard definition of
    optimality of MDP policies (see, e.g., \cite{puterman2005markov}, Ch.~6).
\end{example}

We say that \navy{Bellman's principle of optimality holds}  if $V_\Sigma$ has a
greatest element $v^*$ and, for $\sigma \in \Sigma$,
\begin{equation}\label{eq:bpo}
    \text{$\sigma$ is optimal } \iff 
    \text{ $\sigma$ is $v^*$-greedy}.
\end{equation}

\begin{example}
    The ADP generated by the MDP model in
    Section~\ref{ss:mdps} satisfies Bellman's principle of optimality.
    See, for example, \cite{bertsekas2022abstract}, Lemma~2.1.1 (c).
\end{example}

We say that the \navy{fundamental ADP optimality results} hold for $(V, \TT)$ if
\begin{enumerate}
    \item[(B1)] $V_\Sigma$ has a greatest element $v^*$, 
    \item[(B2)] $v^*$ is the unique solution to the Bellman equation in $V$, and
    \item[(B3)] Bellman's principle of optimality holds. 
\end{enumerate}

When (B1) holds we call the greatest element $v^*$ the \navy{value function}. 
Clearly (B1) is equivalent to the statement that at least one optimal policy
exists.

Properties (B1)--(B3) are not independent, as the next lemma shows.
 
\begin{lemma}\label{l:fo}
    If $V_\Sigma$ has greatest element $v^*$, then 
    $v^*$ satisfies the Bellman equation if and only if
    Bellman's principle of optimality holds.
\end{lemma}

\begin{proof}
    Let $V_\Sigma$ have greatest element $v^*$. Suppose first that $Tv^*=v^*$.
    Fixing $\sigma \in \Sigma$, we claim that that \eqref{eq:bpo} holds.
    As for $\Rightarrow$, if $\sigma \in \Sigma$ is optimal, then $v_\sigma =
    v^*$. Since $T_\sigma \, v_\sigma = v_\sigma$, this implies $T_\sigma \, v^*
    = v^*$. But $T \, v^* = v^*$, so $T_\sigma \, v^* = T v^*$.  It follows
    that $\sigma$ is $v^*$-greedy (by Lemma~\ref{l:torper}).
    As for $\Leftarrow$, if $\sigma$ is $v^*$-greedy, then $T_\sigma \, v^* = T
    v^* = v^*$.   But $v_\sigma$ is the unique fixed point of $T_\sigma$ in $V$, so
    $v_\sigma = v^*$.  Hence $\sigma$ is an optimal policy.  
    As for the converse implication, the definition of greatest elements
    implies existence of a $\sigma \in \Sigma$ such that $v_\sigma =
    v^*$.   By Bellman's principle of optimality, the policy $\sigma$ is
    $v^*$-greedy.  As a result, $T v^* = T_\sigma \, v^* = T_\sigma \, v_\sigma
    = v_\sigma = v^*$.  
\end{proof}

\subsection{Algorithms}

Let $(V, \TT)$ be a regular well-posed ADP with Bellman operator $T$ and
$\sigma$-lifetime value functions $V_\Sigma$.  Suppose that the fundamental
optimality properties (B1)--(B3) hold and let $v^*$ denote the value function.
We consider three major algorithms for computing $v^*$: value function iteration, optimistic policy
iteration and Howard policy iteration. To this end, we define the \navy{Howard policy operator} $H \colon V \to
V_\Sigma$ corresponding to $(V, \TT)$ via
$H v = v_\sigma$  where $\sigma$ is $v$-greedy.
So that $H$ is well-defined, we always select the same $v$-greedy policy when
applying $H$ to $v$. 
Also, fixing arbitrary $m \in \NN$, we define the \navy{optimistic policy
operator} $W_m$ from $V_G$ to $V$ via 
$W_m \, v \coloneq T^m_\sigma v$ where $\sigma$  is $v$-greedy.
As was the case for $H$, we always select a fixed $v$-greedy policy when
applying $W_m$ to $v$. We say that 
\begin{itemize}
    \item \navy{value function iteration (VFI) converges} if $T^n v \uparrow v^*$ for all $v \in V_U$,
    \item \navy{Howard policy iteration (HPI) converges} if $H^n v \uparrow v^*$ for
        all $v \in V_U$, and 
    \item \navy{optimistic policy iteration (OPI) converges} if $W_m^n v
        \uparrow v^*$ for all $v \in V_U$.
\end{itemize}

\section{Optimality Results}\label{s:opres}\label{ss:sr}

In this section we present our main theoretical results.
Proofs are deferred to Section~\ref{s:ir}.
\emph{Throughout this
section, $(V, \TT)$ is a regular well-posed ADP}.  
First,  we state a high-level result that assumes the existence of a fixed point
for the Bellman operator. 

\begin{theorem}\label{t:bk}
    If $(V, \TT)$ is downward stable and $T$ has at least
    one fixed point in $V$, then the fundamental ADP optimal results hold.
\end{theorem}

Now we drop the assumption that $T$ has a fixed point and suppose instead that
$V$ has some form of completeness. 

\begin{theorem}\label{t:bkn}
    If the value space $V$ is chain complete, then the fundamental
    ADP optimal results hold.
\end{theorem}

The next result weakens chain completeness to a milder condition on the space, 
while adding continuity properties on the policy operators. 

\begin{theorem}\label{t:impliesms}
    If $(V, \TT)$ is order continuous and $V$ is countably chain complete, then 
    \begin{enumerate}
        \item the fundamental ADP optimal results hold, and
        \item VFI, OPI, and HPI all converge.
    \end{enumerate}
\end{theorem}

In the previous two theorems, it is assumed that $V$ is order bounded.
The next two theorems drop this assumption.
In the first, we state a result for the case where $(V, \TT)$ is finite, which is
relatively common in applications.  

\begin{theorem}\label{t:bkf}
    If $(V, \TT)$ is finite and order stable, then 
    \begin{enumerate}
        \item the fundamental ADP optimal results hold, and
        \item HPI converges in finitely many steps.
    \end{enumerate}
\end{theorem}

Finally, we consider a setting where $(V, \TT)$ is not finite and the value space
can be unbounded.

\begin{theorem}\label{t:dede}
    Let $(V, \TT)$ be order continuous and order stable. If $V$ is
    countably Dedekind complete and $(V, \TT)$ is bounded above, then
    \begin{enumerate}
        \item the fundamental ADP optimality properties hold and
        \item VFI, OPI and HPI all converge.  
    \end{enumerate}
\end{theorem}

\section{Proofs of Section~\ref{ss:sr} Results}\label{s:ir}

In this section we prove  optimality results from Section~\ref{ss:sr}.
Throughout this section, $(V, \TT)$ is a regular well-posed ADP with Howard
policy operator $H$, optimistic policy operator $W$ and Bellman operator $T$.

\subsection{Preliminaries}

We begin with some lemmas.

\begin{lemma}\label{l:meta}
    The following statements hold. 
    \begin{enumerate}
        \item[\rm{(L1)}] If $v \in V$ with $Hv=v$, then $Tv=v$.
        \item[\rm{(L2)}] The operators $T, W$ and $H$ all map $V_U$ to itself.
        \item[\rm{(L3)}] If $v \in V_U$, then $T v \preceq W v \preceq T^m v$.
    \end{enumerate}
\end{lemma}

\begin{proof}
    As for (L1), fix $v \in V$ with $H \, v = v$ and let $\sigma$ be a
    $v$-greedy policy such that $H \, v = v_\sigma$.  Then $v_\sigma = H \, v = v$. Since
    $\sigma$ is $v$-greedy, $T_\sigma \, v = T \, v$.   Since $v_\sigma$ is
    fixed for $T_\sigma$, we also have $T_\sigma \, v  = v$.  Combining the last
    two equalities proves (L1).

    As for (L2), fix $v \in V_U$.  Since $v \preceq Tv$ and $T$ is order preserving on $V_U$,
    we have $Tv \preceq TTv$.  Hence $Tv \in V_U$.  As for $W$, let
    $\sigma$ be $v$-greedy with $W = T_\sigma^m v$. Since $T$ and $T_\sigma$
    are order preserving and $v \preceq Tv$, 
    we have $W v = T_\sigma T_\sigma^{m-1} v
    \preceq T T_\sigma^{m-1} v \preceq T T_\sigma^{m-1} T v = T T_\sigma^m v =
    TW v$. Hence $W v \in V_U$. Finally,
    regarding $H$, we observe that $Hv \in V_\Sigma$ and, by
    Lemma~\ref{l:vsigvu}, $V_\Sigma \subset V_U$.

    To prove (L3) we fix $v \in V_U$.  Letting
    $\sigma$ be $v$-greedy, we have $v \preceq Tv = T_\sigma \, v$. Iterating on
    this inequality with $T_\sigma$ proves that $(T_\sigma^k \, v)$ is
    increasing. In particular, $Tv = T_\sigma \, v \preceq W v$. 
    For the second inequality in (L3) we use the fact
    that $T_\sigma \preceq T$ on $V$ and $T$ and $T_\sigma$ are both order
    preserving to obtain $W v = T^m_\sigma v \preceq T^m  v$.
\end{proof}

The next lemma adds upward stability and derives additional implications.

\begin{lemma}\label{l:cl2}
    If $(V, \TT)$ is upward stable, then for every $v \in V_U$,
    \begin{equation}\label{eq:rank}
        T^n v \preceq W^n v 
        \qquad \text{and} \qquad
        T^n v \preceq H^n v
    \end{equation}
    for all $n \in \NN$.
    Moreover, the VFI sequence $(T^n v)$, the HPI sequence $(H^n v)$ and the OPI
    sequence $(W^n v)$ are all increasing. 
\end{lemma}

\begin{proof}
    Our first claim is that 
    \begin{equation}\label{eq:uv}
        \text{$u,v \in V_U$ with $u \preceq v$}
        \implies
        Tu \preceq Wv \; \text{ and } Tu \preceq Hv.
    \end{equation}
    To show this we fix such $u, v$ and take $\sigma$ to be $v$-greedy.  Let
    $v_\sigma$ be the $\sigma$-value function, so that $T_\sigma \, v_\sigma =
    v_\sigma$ and $v_\sigma = H v$.  Since $v \in V_U$ we have
    \begin{equation}\label{eq:vhv}
        v 
        \preceq Tv 
        = T_\sigma \, v 
        \preceq T^m_\sigma \, v 
        = W v 
        \preceq v_\sigma
        = Hv.
    \end{equation}
    The second inequality is by iterating on $v \preceq T_\sigma \, v$, while
    the third is by upward stability.   Since $Tu \preceq Tv$, we can use
    \eqref{eq:vhv} to obtain \eqref{eq:uv}.  Iterating on \eqref{eq:uv} produces
    \eqref{eq:rank}.  The last claim in Lemma~\ref{l:cl2} follows from
    \eqref{eq:vhv}, which tells us that elements of $V_U$ are mapped up by $T$,
    $W$, and $H$.
\end{proof}

\begin{corollary}\label{c:cl2}
   If $(V, \TT)$ is upward stable and an optimal policy exists,  
   then convergence of VFI implies convergence of OPI and convergence of HPI.
\end{corollary}

\begin{proof}
    Assume the conditions of the corollary and fix $v \in V_U$.
    Since an optimal policy exists, $v^*$ exists and is the greatest element of $V_\Sigma$. 
    Lemma~\ref{l:cl2} yields $v \preceq T^n v \preceq W^n v \preceq v^*$ for all
    $n$, where the last inequality follows from \eqref{eq:vhv} and the fact that
    $v_\sigma \preceq v^*$ for all $\sigma$.  Hence convergence of VFI implies
    convergence of OPI.  The proof for HPI is similar.
\end{proof}

\subsection{Remaining Proofs}

We now prove the main optimality results from Section~\ref{ss:sr}. 

\begin{proof}[Proof of Theorem~\ref{t:bk}]
    Let $(V, \TT)$ be downward stable and suppose that $T$ has at least one
    fixed point $\bar v$ in $V$. Since $(V, \TT)$ is regular, there exists a
    $\sigma \in \Sigma$ with $T_\sigma \, \bar v = T \bar v$
    (Lemma~\ref{l:torper}).  Since $(V, \TT)$ is well-posed, this last equality
    and $T \bar v = \bar v$ imply that $\bar v$ is the unique fixed point of
    $T_\sigma$.  Thus, $\bar v \in V_\Sigma$.  Moreover, if $\tau \in \Sigma$,
    then $T_\tau \, \bar v \preceq T \bar v = \bar v$, so, by downward
    stability, $v_\tau \preceq \bar v$.  Hence $\bar v$ is
    both the greatest element of $V_\Sigma$ and a solution to the Bellman
    equation in $V$. The fundamental ADP optimality properties now follow from 
    Lemma~\ref{l:fo}.
\end{proof}

\begin{proof}[Proof of Theorem~\ref{t:bkn}]
    Let $V$ be chain complete. $(V, \TT)$ is order stable by Lemma~\ref{l:ocius} and, by
    Theorem~\ref{t:tk}, $T$ has at least one fixed point in $V$.
    Hence the conditions of Theorem~\ref{t:bk} hold, which
    implies the fundamental ADP optimality properties.
\end{proof}

\begin{proof}[Proof of Theorem~\ref{t:impliesms}]
    Suppose $(V, \TT)$ is order continuous and $V$ is countably chain complete.
    Since $T$ is order continuous (Lemma~\ref{l:ostst}), Theorem~\ref{t:tk}
    implies that $T$ has at least one fixed point in $V$. Also, by
    Lemma~\ref{l:ocius}, $(V, \TT)$ is order stable.  Hence, by
    Theorem~\ref{t:bk}, the fundamental ADP optimality properties hold.
    Moreover, for $v \in V_U$ the sequence $v_n \coloneq T^n v$ is increasing.
    Since $T$ is order continuous, the supremum is a fixed point of 
    $T$ (Theorem~\ref{t:tk}).  But, by (B2) of the fundamental ADP optimality
    properties, the value function $v^*$ is the only fixed point of $T$ in $V$.
    Hence VFI converges.  Convergence of OPI and HPI now follows from Corollary~\ref{c:cl2}.
\end{proof}

\begin{proof}[Proof of Theorem~\ref{t:bkf}]
    Let $(V, \TT)$ be order stable and finite. Fix $v \in V$ and let $v_n = H^n
    v$ for all $n \in \NN$. By Lemma~\ref{l:cl2}, $v_n \preceq v_{n+1}$
    for all $n \in \NN$. Since $(v_n)$ is contained in the finite set
    $V_\Sigma$, it must be that $v_{n+1} = v_n$ for some $n \in \NN$. But then
    $H \, v_n =  v_n$, so, by Lemma~\ref{l:meta}, we have $v_n$ is a fixed point
    of $T$.  Hence, by Theorem~\ref{t:bk}, the fundamental ADP optimality
    properties hold.  By these same properties, the fixed point $v_n$ equals the value function $v^*$.
    Thus, we have also shown that HPI converges in finitely many steps.
\end{proof}

\begin{proof}[Proof of Theorem~\ref{t:dede}]
    In view of Theorem~\ref{t:bk}, the fundamental ADP optimality
    properties will hold when $T$ has a fixed point in $V$.  To see that this is
    true,  fix any $v \in V_U$ (which is nonempty by Lemma~\ref{l:vsigvu})
    and set $v_n \coloneq T^n v$.  Since $(V, \TT)$ is bounded above and $V$ is
    countably Dedekind complete, there exists a $\bar v \in V$ with $v_n \uparrow
    \bar v$.  We claim that $T \bar v = \bar v$.  Indeed, $v_{n+1} = T v_n
    \preceq T \bar v$ for all $n$, so, taking the supremum, $\bar v \preceq T
    \bar v$.  For the reverse inequality we take $\sigma$ to be $\bar v$-greedy
    and use order continuity of $T_\sigma$ to obtain
    \begin{equation*}
        T \bar v 
        = T_\sigma \, \bar v
        = T_\sigma \, \bigvee_n v_n
        = \bigvee_n T_\sigma \, v_n
        \preceq \bigvee_n T \, v_n
        = \bigvee_n v_{n+1}
        = \bar v.
    \end{equation*}
    The fundamental ADP optimality properties are now proved.
    In view of these properties, the only fixed point of
    $T$ in $V$ is $v^*$. Hence $T^n v = v_n \uparrow \bar v = v^*$.
    This proves convergence of VFI. Convergence of OPI and HPI follow from Corollary~\ref{c:cl2}.
\end{proof}

\section{Applications}\label{s:apps}

Next we illustrate  how the  ADP optimality results stated above can be
 applied.

\subsection{Non-EU Discrete Choice}\label{ss:noneu}

Some studies have found incompatibilities between data and predictions of
utility maximization problems
founded on additively separable preferences (see, e.g, \cite{lu2023does}).
To further this line of analysis, we return to the discrete choice 
Bellman equation in Section~\ref{ss:struct}, which was motivated by
structural estimation, while replacing ordinary conditional expectation with a
general certainty equivalent operator (so that preferences can fail to be
additively separable).  In particular,
we adopt the setting and assumptions of Section~\ref{ss:struct} while modifying the
policy operator \eqref{eq:ruts} to $(T_\sigma \, g)(x, a) = (\eE H_\sigma \,
g)(x, a)$, where
\begin{equation*}
    (H_\sigma \, g)(x') \coloneq
             \int \left\{
                 r(x', \sigma(x'), e') + \beta g(x', \sigma(x')) 
             \right\}
             \nu(\diff e')
\end{equation*}
and $\eE$ is a certainty equivalent operator mapping $\RR^\Xsf$ into $\RR^\Gsf$. This
means that $\eE$ is order preserving with respect to the pointwise order and $\eE \lambda = \lambda$ whenever
$\lambda$ is constant.  (Thus, $\eE$ is a generalization of a conditional
expectations operator.)   We call $\eE$ \navy{constant subadditive} if $\eE(f +
\lambda) \leq \eE f + \lambda$ for all $f \in \RR^\Xsf$ and $\lambda \in \RR_+$.

Let $\TT = \{T_\sigma\}_{\sigma \in \Sigma}$.  Each $T_\sigma$ is order
preserving on $\RR^\Gsf$ under the usual pointwise order, so $(\RR^\Gsf, \TT)$
is an ADP.  Moreover, given $g \in \RR^\Gsf$, a policy $\sigma \in \Sigma$ is
$g$-greedy whenever 
\begin{equation*}
    \sigma(x) \in \argmax_{a' \in \Asf}
        \int
            \left[
                r(x', a', e') + \beta g(x', a')
            \right]
            \nu(\diff e')
            \quad \text{for all } x \in \Xsf.
\end{equation*}
Since $\Asf$ is finite, such a policy always exists. Hence $(\RR^\Gsf, \TT)$ is
regular.

\begin{proposition}\label{p:stor}
    If $\eE$ is constant subadditive, then the fundamental ADP optimality
    properties hold and HPI converges in finitely many steps.
\end{proposition}

\begin{proof}
    Since $(\RR^\Gsf, \TT)$ is regular and finite, it suffices to show that 
    $(\RR^\Gsf, \TT)$ is also order stable (by Theorem~\ref{t:bkf}). To this
    end, fix $f, g \in \RR^\Gsf$. Since $\eE$ and $H_\sigma$ are order preserving,
    we have
    $T_\sigma \, f 
        = \eE H_\sigma (g + f - g)
        \leq \eE H_\sigma (g + \| f - g\|)
        \leq \eE (H_\sigma \, g + \beta \| f - g\|)$.
    Using constant subadditivity of $\eE$ and rearranging gives
    $T_\sigma \, f - T_\sigma \, g \leq \eE \beta \| f - g\| = \beta \| f - g\|$.
    Reversing the roles of $f$ and $g$ gives $|T_\sigma \, f - T_\sigma \, g | \leq
    \beta \| f - g\|$, so each $T_\sigma \in \TT$ is a contraction on $\RR^\Gsf$.
    Since $\RR^\Gsf$ is complete under the supremum norm, $(\RR^\Gsf, \TT)$ is
    well-posed.  Moreover, from the argument in Example~\ref{eg:pospace},
    $(\RR^\Gsf, \TT)$ is order stable.  This completes the proof of
    Proposition~\ref{p:stor}.
\end{proof}

As an illustration, suppose that $\eE$ is the risk-sensitive certainty
equivalent
\begin{equation*}
    (\eE f)(x, a)
    \coloneq
    \frac{1}{\theta}
        \ln 
        \left\{ 
                \int \exp \left[ \theta f(x')
            \right] P(x, a, \diff x')
        \right\}
        \qquad ((x, a) \in \Gsf)),
\end{equation*}
where $P$ is a stochastic kernel from $\Gsf$ to $\Xsf$ and $\theta$ is a nonzero
constant. This choice of certainty equivalent is constant subadditive, so,
 among other things, Proposition~\ref{p:stor} tells us that
$\sigma \in \Sigma$ is optimal if and only if
$\sigma(x) \in \argmax_{a' \in \Asf}
            \left[
                r(x', a') + \beta g^*(x', a')
            \right]$ for all $x \in \Xsf$,
where $g^*$ is the unique solution to the functional equation
\begin{equation*}
    g(x, a) =
    \frac{1}{\theta}
        \ln 
        \left\{ 
        \int
           \exp \left\{ 
                \theta \max_{a' \in \Asf}
                \left[
                    r(x', a') + \beta g(x', a')
                \right]
            \right\} P(x, a, \diff x')
    \right\}
\end{equation*}
in the value space $\RR^\Gsf$.\footnote{Another example of a nonlinear certainty equivalent operator is the quantile
operator studied in \cite{de2019dynamic}, which allows
for separation of intertemporal elasticity of substitution and risk aversion.
This certainty equivalent is also constant subadditive, so
Proposition~\ref{p:stor} extends the results in \cite{de2019dynamic}.}

\subsection{Firm Valuation}\label{ss:firmex}

We consider a firm valuation problem studied by
\cite{jovanovic1982selection} with the following extensions: (i) firm profits depend on an aggregate shock, as
well as a firm-specific shock and a cross-sectional distribution,
 (ii) the interest rate is allowed to vary over time, (iii) the outside
option of the firm is permitted to depend on aggregates and the cross-section,
and shocks and rewards are allowed to be discontinuous and unbounded.

In this version of the problem,
a firm that receives current profit $\pi(s, \mu, z)$ and then transitions to
the next period, where management will choose to either exit and receive 
$q(\mu', z')$ or continue. Thus, the maximal expected firm value $v(s, \mu, z)$ obeys
\begin{equation}\label{eq:febell}
    v(s, \mu, z) = \pi(s, \mu, z) 
    + \beta(\mu, z) \, \EE_{(s, \mu, z)} \max\{q(\mu', z'),  v(s', \mu', z') \}.
\end{equation}
Here $s$ is an idiosyncratic state for the firm that takes values in set $S$,
$\mu$ is a cross-sectional distribution taking values in a space
$D$, $z$ is an aggregate shock taking values in set $Z$, and $\pi(s, \mu, z)$ is
current profit. Primes denote next period values.   The discount factor $\beta$
depends on cost of capital and hence the current state.
Let $x \coloneq (s, \mu, z)$ take values in $\Xsf \coloneq S \times D \times Z$. Let
$\bB$ be a $\sigma$-algebra over $\Xsf$ that makes
$\pi, \beta$ and the transition probabilities measurable. We rewrite the
dynamics as $x' \sim P(x, \cdot)$, meaning that $P$ is a stochastic kernel on $(\Xsf, \bB)$ and the next
period composite state $x'$ is drawn from distribution $P(x, \cdot)$.    With
this notation, \eqref{eq:febell} becomes
\begin{equation}\label{eq:fe_modbell}
    v(x) = \pi(x) + \beta(x) 
        \int \max \left\{q(x'), v(x') \right\} P(x, \diff x') 
        \qquad (x \in \Xsf).
\end{equation}
A $v$ that solves \eqref{eq:fe_modbell} gives firm valuation at each state under 
optimal management. 

Let $K$ be the \navy{discount operator} defined by
\begin{equation*}
    (Kv)(x) \coloneq \beta(x) \int v(x') P(x, \diff x').
\end{equation*}
We suppose there exists a $\sigma$-finite measure $\phi$ on $(\Xsf, \bB)$ such
that $\pi, q$ and $\beta$  are nonnegative elements of $L_1 \coloneq L_1(\Xsf,
\bB, \phi)$, and that $K$ maps $L_1$ to itself.     We endow $L_1$ with the
$\phi$-a.e.\ pointwise order $\leq$, so that $f \leq g$ means $\phi \{f > g\} =
0$. In what follows, for any linear operator $A$ on $L_1$, we use $\rho(A)$ to
represent the spectral radius of $A$. Also, $A$ is called positive if $0 \leq v$
implies $0 \leq Av$.

\begin{assumption}\label{a:firms}
     The discount operator obeys $\rho(K) < 1$.
\end{assumption}

Assumption~\ref{a:firms} is weaker than that found in \cite{hansen2012recursive}
and related sources, since we impose no irreducibility or compactness conditions
on $K$. (Later, in Proposition~\ref{p:firm_nec}, we show that, when such
conditions \emph{are} in force, Assumption~\ref{a:firms} is both necessary and
sufficient for optimality.)

Let $\Sigma$ be the set of policies, each of which is a
$\bB$-measurable map $\sigma$ from $\Xsf$ to $\{0,1\}$. Here $\sigma(x) = 1$
indicates the decision to exit at state $x$ and $\sigma(x) = 0$ indicates the
decision to continue. To each $\sigma \in \Sigma$ we assign the policy operator
\begin{equation}\label{eq:fppol}
    T_\sigma \, v = \pi + K (\sigma q + (1 - \sigma) v)
\end{equation}
Since $K$ is positive and hence order preserving,  $T_\sigma$ is 
order preserving on $L_1$. Hence $(L_1, \TT)$ is an ADP when $\TT \coloneq \setntn{T_\sigma}{\sigma \in \Sigma}$.

Let $V$ be all $v \in L_1$ such that $0 \leq
v \leq \bar v$, where
$\bar v \coloneq (I - K)^{-1}(\pi + K q)$
and $I$ is the identity map ($\bar v$ is well-defined by
Assumption~\ref{a:firms}).  Straightforward arguments show
that every $T_\sigma$ maps $V$ to itself.  Hence $(V, \TT)$ is also an ADP.
Since $0 \leq K (1 - \sigma) \leq K$ we have $\rho(K(1 - \sigma)) \leq \rho(K) <
1$, so each $T_\sigma$ is has a unique fixed point $v_\sigma$ in $V$.
In particular, $(V, \TT)$ is well-posed.
By definition, its ADP Bellman operator obeys $Tv \coloneq
\bigvee_{\sigma \in \Sigma} T_\sigma \, v = \pi + K (q \vee v)$, which coincides
with \eqref{eq:fe_modbell}.  This means that solving the ADP optimization
problem is equivalent to solving the original dynamic program with Bellman
operator \eqref{eq:fe_modbell}.

\begin{proposition}\label{p:firms}
    If Assumption~\ref{a:firms} holds, then the fundamental ADP optimality
    properties hold and VFI, OPI, and HPI all converge.
\end{proposition}

\begin{proof}
    Given $v \in V$, the set $\setntn{T_\sigma \, v}{\sigma \in \Sigma}$ has greatest
    element $\pi + K (q \vee v)$, which is attained by the $v$-greedy policy $\sigma
    = \1\{q \geq v\}$.  Hence $(V, \TT)$ is regular.  Also, $K$ is
    order continuous because positive linear operators on $L_1$ are order
    continuous (see, e.g., \cite{zaanen2012introduction}, Example~21.6).  It then follows
    that each $T_\sigma$ is order continuous (since the order limit $\uparrow$
    is preserved under basic arithmetic operations -- see, e.g., Theorem~10.2 of
    \cite{zaanen2012introduction}) and, in particular, $(V, \TT)$ is
    order continuous.  Because $(V, \TT)$ is regular, well-posed and order
    continuous, and because $V$ is chain complete (see, e.g., Example~12.5 of
    \cite{zaanen2012introduction}), Theorem~\ref{t:impliesms}
    applies.  This yields the conclusions of Proposition~\ref{p:firms}.
\end{proof}

\begin{figure}
   \centering
   \scalebox{0.6}{ \includegraphics[trim = 0mm 2mm 0mm 0mm, clip]{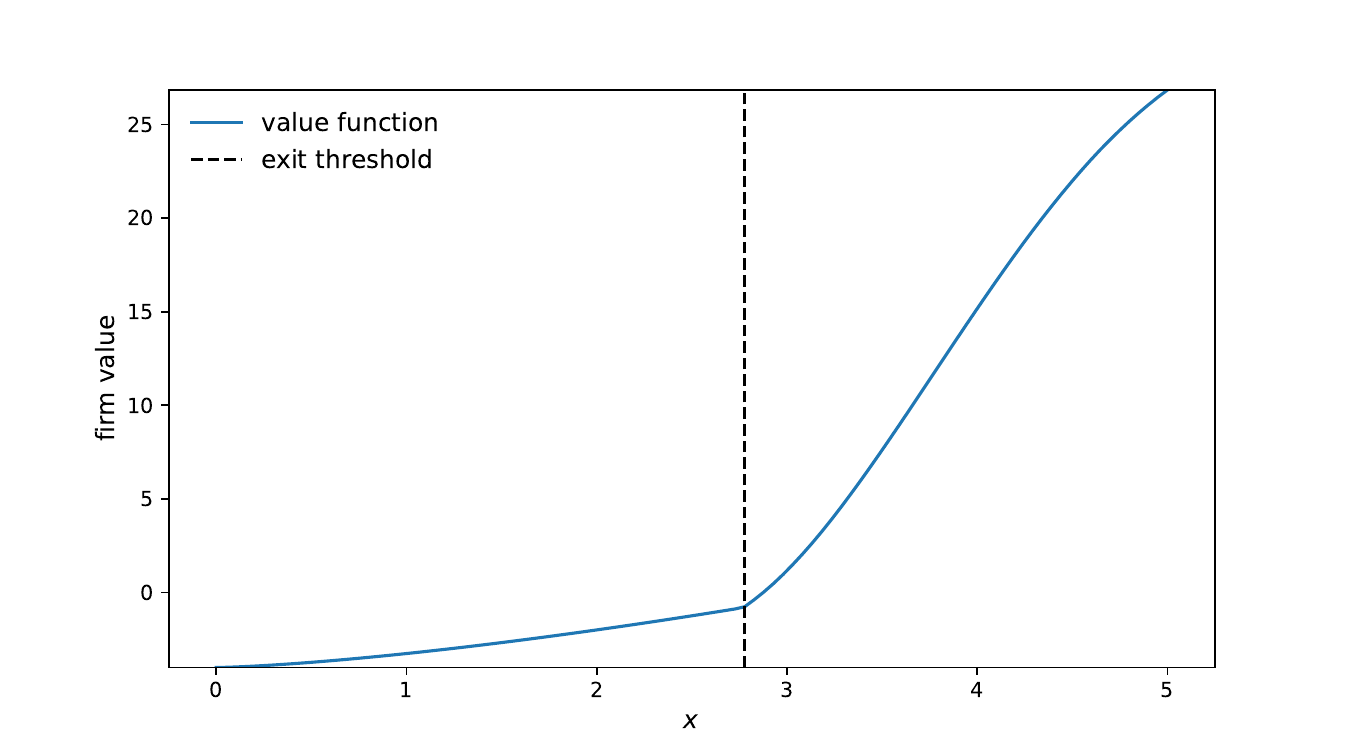}}
   \caption{\label{f:firm_exit} Firm value function and exit threshold}
\end{figure}

Figure~\ref{f:firm_exit} shows an approximation of the value function $v^*$
computed by VFI, as well as a representation of a $v^*$-greedy policy $\sigma$ in the form of an exit threshhold.  For
$x$ below the threshold, $\sigma(x)=0$, indicating that exit is optimal.
In this example, the state space is just $\RR_+$ and $x$ can be thought of as
productivity.  The function $\pi$ is given by 
\begin{equation*}
    \pi(x) = \max_{\ell \geq 0} \left\{p x \ell^\theta - c - \ell \right\},
\end{equation*}
where $p$ is the output price, $\ell$ represents labor input, $\theta$ is a
productivity parameter, and $c$ is a fixed cost.  The dynamics of $x$ are given
by $P(x, \diff x) \eqdist A x$, where $A$ is lognormal $(-0.012, 0.1)$. We set
$\beta = 0.95$, $\theta=0.3$, and $c=4$.  The outside option $q$ is set to zero.

To show that our assumptions are weak, we now prove that, under some mild
conditions, well-posedness fails whenever Assumption~\ref{a:firms} fails.  This
means that Assumption~\ref{a:firms} is necessary for the dynamic program to be
well-defined.

\begin{proposition}\label{p:firm_nec}
    Let $\pi$ be nonzero and let $K$ be weakly compact and irreducible on $L_1$.
    In this setting, if $(V, \TT)$ is well-posed, then Assumption~\ref{a:firms}
    holds.
\end{proposition}

\begin{proof}
    Let $\pi$ be nonzero and let $K$ be weakly compact and irreducible.
    Let $K'$ be the adjoint of $K$ and let $\lambda$ be the spectral radius.  By the Krein--Rutman theorem
    (see, in particular, Lemma~4.2.11 of \cite{meyer2012banach}), there exists
    an $e \in L_\infty$ such that $K' e = \lambda e$ and, in addition,
    $\inner{e, f} > 0$ for all nonzero nonnegative $f \in L_1$. Consider the policy $\sigma \equiv 0$.  Under
    this policy we have $T_\sigma \, v = \pi + Kv$. If $(V, \TT)$ is well-posed,
    then there exists a solution $v \in V$ to $v = \pi + K v$ in $V$. Since
    $\pi$ is nonzero and $v = \pi + K v \geq \pi$, the same is true for $v$.
    Now observe that
    $\inner{e, v} 
        = \inner{e, \pi} + \inner{e, Kv}  
        = \inner{e, \pi} + \inner{K' e, v}  
        = \inner{e, \pi} + \lambda \inner{e, v}$.
    Since $v$ and $\pi$ are nonnegative and nonzero, it must be that $\inner{e,
    \pi} > 0$ and $\inner{e, v} > 0$.  Therefore $\lambda$ satisfies $(1-\lambda) \alpha = \beta$ for 
    $\alpha, \beta > 0$.  Hence $\lambda < 1$.
\end{proof}

\section{Conclusion}\label{s:con}

The framework constructed in this paper represents dynamic
programs as operators over partially ordered sets and allows us to acquire  a
range of new optimality results that include many existing results as special
cases.  These methods are suitable for  applications with a number of challenging
features.

A limitation of our results is that we assumed our dynamic
programs are regular. Some dynamic programs do not have this property because
certain policies lead to infinite loss (see, e.g., \cite{li2024exact},
\cite{pates2024optimal}, or
Chapters~3--4 of \cite{bertsekas2022abstract}). Others lack this property due to
nonstandard discounting \citep{balbus2020markov, jaskiewicz2021markov}.
Extensions of the results in this paper to such problems would be valuable.

Another of our  assumptions that could be altered  is well-posedness, i.e.,that each $T_\sigma$ has a
unique fixed point $v_\sigma$.  It  could be replaced by
generalizing the approach in \cite{bertsekas2022abstract}, where $v_\sigma$ is defined by $v_\sigma(x)
\coloneq \limsup_k T_\sigma^k \, \bar v(x)$ for some fixed reference point $\bar v
\in V$. The limsup could be generalized to an abstract partially ordered set
environment by setting $v_\sigma \coloneq \wedge_{n \geq 1} \, \vee_{k
\geq n} T_\sigma^k \bar v$. The element $v_\sigma$ would always be
well-defined if, say, $V$ is a complete lattice.  We have not yet explored this
modification of our  framework,  but think it would be worthwhile. % some researchers might find it advantageous.

We have focused on applications and theoretical settings where
optimal policies always exist. If one wishes to consider approximately optimal
policies, then some metric on the value space must be added in order to measure
approximations. A promising  path forward would be to replace the assumption that
$V$ is an arbitrary partially ordered set with the assumption that $V$ is a
partially ordered space; that is, a metric space with partial order $\preceq$
such that the order $\preceq$ is preserved under limits. 

Many further extensions  could be built on top of our framework.  One example is average-cost optimality for dynamic
programs, which we have not considered.  Another is continuous time models.
These topics are also left for future work.

% \appendix
%
% \section{Notes}
%
% The space of bounded Borel measurable functions and the $L_p$ spaces are either
% Dedekind complete or Dedekind $\sigma$-complete (\cite{meyer2012banach}, p.
% 8--9).
%
% The following paragraph refers to order continuity in Riesz spaces.
%
% The norm on $L_p$ is order continuous when $p < \infty$, and this implies other
% nice properties.  See \cite{meyer2012banach}, p. 86.  Also, when $p < \infty$,
% every positive linear operator from $L_p$ to itself is order continuous.
% See \cite{zaanen2012introduction}, p. 147.
%

\bibliographystyle{apalike}
\bibliography{../qe_bib}

\end{document}